\documentclass[11pt]{article}
\usepackage{pdfsync}
\usepackage{amsmath,amsthm,amsfonts,amssymb,mathrsfs,bm,wasysym}
\usepackage{epsfig}
\usepackage{color}
\usepackage{verbatim}
\usepackage{multicol}
\usepackage{comment}
\usepackage{graphicx}
\usepackage{centernot}
\usepackage{tikz}
\usepackage{color}
\usepackage{enumerate}
\usepackage{comment}
\usepackage{thmtools}
\usepackage{float}
\usepackage[figurename=Figure]{caption}
\usepackage{subcaption}
\usepackage{textcomp}
\usepackage[normalem]{ulem}
\usepackage[utf8]{inputenc} 
\usepackage[T1]{fontenc}
\usepackage{hyperref}
\usepackage{titlesec}
\usepackage{chngcntr}
\usepackage{cleveref}

\usetikzlibrary{graphs.standard, quotes, calc, patterns,decorations.pathreplacing, decorations.markings, positioning, calligraphy}

\parskip \medskipamount
\titlespacing*{\subsubsection}
{0pt}{0.4\baselineskip}{0.2\baselineskip}

\declaretheoremstyle[
qed={//}
]{defstyle}

\newtheorem{theorem}{Theorem}[section]
\declaretheorem[style=defstyle,sibling=theorem]{definition}
\newtheorem{lemma}[theorem]{Lemma}
\newtheorem{proposition}[theorem]{Proposition}

\newtheorem{claim}[theorem]{Claim}

\counterwithout{equation}{section}

\newcommand{\pr}[1]{\mathbb{P}\!\left(#1\right)}
\newcommand{\prstart}[2]{\mathbb{P}_{#2}\!\left(#1\right)}

\newcommand{\E}[1]{\mathbb{E}\!\left[#1\right]}

\DeclareMathOperator{\Var}{Var}

\newcommand{\Z}{\mathbb{Z}}

\newcommand{\Znonneg}{\Z_{\ge0}}

\newcommand{\til}[1]{\widetilde{#1}}

\newcommand{\trel}{t_{\mathrm{rel}}}

\newcommand{\tmix}[1]{t_{\mathrm{mix}}\left(#1\right)}

\newcommand{\eps}{\varepsilon}

\newcommand{\tmixtext}{t_{\mathrm{mix}}}

\newcommand{\Pois}[1]{\mathrm{Pois}\!\left(#1\right)}

\newcommand{\Ber}[1]{\mathrm{Ber}\!\left(#1\right)}

\newcommand{\dtv}[2]{d_{\mathrm{TV}}\left({#1},{#2}\right)}

\newcommand{\diam}[1]{\mathrm{diam}\left(#1\right)}

\DeclareMathOperator{\numcomps}{\#comps}
\DeclareMathOperator{\comps}{comps}

\begin{document}

\title{Mixing time of the random walk on the giant component of the random geometric graph}

\author{
Magnus H. Haaland
\thanks{University of Cambridge, Cambridge, UK. E-mail: {mhh36@cam.ac.uk}}
\and An{\dj}ela \v{S}arkovi\'c
\thanks{King's College, University of Cambridge, UK. E-mail: {as2572@cam.ac.uk}}
}
\date{}

\maketitle

\begin{abstract} We consider a random geometric graph obtained by placing a Poisson point process of intensity $1$ in the $d$-dimensional torus of side length $n^{\frac{1}{d}}$ and connecting two points by an edge if their distance is at most $r$. We consider the case of $d\ge 2$ and $r\in [r_{\min}, r_{\max}]$, where $r_{\min}<r_{\max}$ are any constants with $r_{\min}>r_g$ and $r_g$ is a constant above which this graph has a giant component with high probability. We show that, with high probability, the mixing time and the relaxation time of the simple random walk on the giant component in this case are both of order $n^{\frac{2}{d}}$ and that therefore there is no cutoff. We also obtain bounds for the isoperimetric profile of subsets of the giant component of at least polylogarithmic size.
\end{abstract}

\section{Introduction}

In this paper, we study the mixing time of a simple random walk on a random geometric graph. Before stating our results, we start by defining the model.

\begin{definition}[Random Geometric Graph (RGG) on the torus]\label{def:RGG_torus} Let \(\Lambda_n^{(d)}\) be the torus on $\left[\frac{-\sqrt[d]n}2, \frac{\sqrt[d] n}2\right]^d$ with the Euclidean metric $\|\cdot-\cdot\|_2$.
A random geometric graph (RGG) is a random graph \(G=G\left(\Lambda_n^{(d)},\,r(n)\right)\) obtained as follows: Let $V_n$, the set of vertices of $G_n$, be a Poisson Point Process (PPP) on \(\Lambda_n^{(d)}\) of intensity $1$, and let the set of edges be $E_n=\{\{x,y\}: x,y\in V_n, \|x-y\|_2\le r(n)\}$, i.e.\ edges are added between vertices that are within distance at most \(r(n)\) from each other.
\end{definition}

It is known that there exists \(r_g\) such that if \(r(n) \ge (1+\eps)r_g\) for some $\eps>0$, then with high probability \(G\) has a unique giant component of size of order $n$ (see \cite[page 149]{penrose_rgg}).
We denote this giant component by $L_1(G)$ and we study the behaviour of a random walk on the giant in the supercritical regime for $r(n)$ of order $1$.
We next recall the  definitions of mixing and relaxation times.

\begin{definition}\label{def:tmix_cutoff}
Let $X$ be a Markov chain on a finite state space $V$, with a unique invariant distribution~$\pi$. For $\eps\in(0,1)$, we define the $\eps$-mixing time of $X$ as
\[
\tmix{\eps}:=\quad\max_{x\in V}\:\inf\left\{t\ge0:\:\dtv{\prstart{X_t=\cdot}{x}}{\pi(\cdot)}\le\eps\right\}.
\]
where $d_{\mathrm{TV}}$ denotes the total variation distance between two distributions, defined as 
\[
\dtv{\mu}{\nu}=\frac12\sum_{x}|\mu(x)-\nu(x)|
\]
for $\mu$ and $\nu$ probability distributions. 
We write $\tmixtext=\tmix{1/4}$.
%We say that a sequence $\left(X^{(n)}\right)$ of Markov chains exhibits cutoff if for all $\varepsilon \in (0,1)$ we have \begin{equation}\label{eq:cutoff_def} \lim_{n\to \infty}\frac{\tmix{1-\varepsilon}}{\tmix{\varepsilon}}=1.\end{equation}
\end{definition}

We say that an event $A$ holds with high probability as $n\to\infty$, if $\pr{A}\to 1$ as $n\to\infty$. 
Our main result states the following. 
%[Mixing time of RGG for $d\ge 2$, $r_{\min}\le r(n)\le r_{\max}$]
\begin{theorem}\label{trm:tmixRGG} 
Let $d\geq 2$ and let $r_g<r_{\min}\le r_{\max}$. Then there exist positive constants $c$ and $C$ so that  with high probability as $n\to\infty$ the mixing time of the simple random walk and of the lazy simple random walk on the giant component $L_1(G)$  satisfies 
\[
c n^{\frac2d}\le \tmixtext\le C n^{\frac{2}{d}}.
\] 
\end{theorem}

The absolute relaxation time of a reversible chain is defined to be $\trel=\max\{\frac{1}{1-\lambda_2}, \frac{1}{1-|\lambda_n|}\}$, where $\lambda_2$ is the second largest eigenvalue of the transition matrix of the Markov chain and $\lambda_n$ is the smallest one. We also establish the following.

\begin{proposition}\label{prop:trelRGG} 
Let $d\geq 2$ and let $r_g<r_{\min}\le r_{\max}$. Then there exist positive constants $c$ and $C$ so that  with high probability as $n\to\infty$ the absolute relaxation time of the simple random walk and of the lazy simple random walk on the giant component $L_1(G)$ satisfies 
\[c n^{\frac2d}\le \trel\le C n^{\frac{2}{d}}.\]
\end{proposition}

This work was a product of a summer project, which lasted from June to September of 2025. Just before posting the write-up, we were made aware that the paper \cite{RGGmixing} showing the same results (for the continuous time setup) has appeared online recently, after the conclusion of our project. 

\section{Proof overview and preliminary results}

We obtain the upper bound on the mixing time, and therefore the relaxation time, by getting a bound on the isoperimetric profile defined as follows for a Markov chain with transition matrix $P$ and invariant distribution $\pi$. 

\begin{definition}\label{def:isop_prof} Let the conductance of a subset $A$ of the state space be $\Phi(A)=\frac{Q(A,A^c)}{\pi(A)}$, where for sets $A$ and $B$ we define $Q(A,B)=\sum_{x\in A, y\in B}\pi(x)P(x,y)$. The isoperimetric profile is a function $\Phi_*:[\pi_{\min}:=\min_{x}{\pi(x)}, \infty)\to \mathbb{R}$  given by 
\[\Phi_*(r)=\inf_{A:\, 0<\pi(A)\le \min\{r,\frac{1}{2}\}}\Phi(A).\vspace{-0.7cm}\]
\end{definition}

To upper bound the mixing time we use that for a lazy chain
$\tmixtext\le C \int_{4{\pi_{\min}}}^{8}\frac{du}{u\Phi^2_*(u)}$ for a positive constant $C$, which was established by Morris and Peres in \cite{evolvingsets} using the evolving sets method.

Therefore, we can bound the mixing time using this result and the following proposition, which we establish in Section~\ref{sec:isopbounds} and which is one of the main technical contributions of this work. 

\begin{proposition}\label{prop:largeconnsetshavelargeboundaryanyd}
In the setup of Theorem~\ref{trm:tmixRGG}, there exist constants $C$ and $c$ such that the following holds with high probability.
For all $A\subset L_1(G)$ with both $A$ and $L_1(G)\setminus A$ connected and with $\frac12\ge\pi(A)\ge \frac{C(\log n)^{2d^2}}{n}$
the number of edges between \(A\) and \(A^c\) is at least $c(\deg(A))^{1-\frac{1}{d}}$.
\end{proposition}

The most challenging part of this work is obtaining the above bound.
Our proof is motivated by the work of Pete \cite{pete_exp_cluster_rep}, who obtains analogous results for supercritical percolation by first showing exponential cluster repulsion, i.e.\ that the probability that supercritical percolation in $\mathbb{Z}^d$ has a cluster $A$ of size $m$ containing zero which is connected by at least $t$ closed edges to another cluster $B$ of size at least $m$ is at most $\exp(-c\max(m^{1-\frac{1}{d}},t))$ for a suitable constant $c$.
This is obtained by exhibiting from such $A$ and $B$ a large connected set of blocks which are bad, where a block is a box of side length $\frac{3L}{2}$ with centre in $(L\Z)^d$, and a block is bad if the restriction of the percolation to it has zero or at least two suitably big clusters, or one suitably big cluster that does not reach all the $(d-1)$-dimensional faces of the block.
It is not hard to prove that when $L$ is a large enough constant, having a connected set of bad blocks near zero is exponentially unlikely in the size of the set, and the repulsion result follows by considering the blocks of a suitably defined boundary of $A$ together with the blocks suitably intersecting both $A$ and $B$.
Upon showing the exponential cluster repulsion, the statement analogous to \Cref{prop:largeconnsetshavelargeboundaryanyd} about the boundary of subsets in the giant follows by, for each $t$, considering connected sets $A$ and $A^c$ which have a total of $t$ edges between them, and closing off the open ones among these.
This gives two large clusters with at least $t$ closed edges connecting them, and a suitable computation which accounts for the exponential cost of closing the edges, but also that it is exponentially unlikely to have such two clusters, gives the result.
Our proof aims to show the analogous, but unlike in the percolation setup, we do not know that we only have points in each integer coordinate, and all of the points which are close enough are connected, and so we can not have clusters with closed edges between them.
Instead, we aim to delete the vertices in $A^c$ which are connected to $A$ by a direct edge and to think of $t$ as the number of boxes in a grid of side length between $r$ and $2r$ which have a vertex of $A$ and, in itself or a suitably nearby box, a vertex of $A^c$.
If the boundary between $A$ and $A^c$ is not large, after this deletion, we are still left with at least a constant proportion of vertices in $A^c$, but we could end up with many connected components of $A^c$, while in the case of percolation, after closing the edges $A^c$ stays connected.
This means that we need to upgrade the cluster repulsion result to a case of one cluster $A$, which is suitably close to possibly many different clusters whose union is some large enough set $B$.
This is done in~\Cref{prop:exponentialclusterrepulsionmodified}, where such a result can be obtained only when all of the clusters forming $B$ are suitably large.
Similarly to the percolation case, we aim to obtain this result by finding a large set of bad blocks, which in our case does not need to be connected, but instead has an upper bound on the number of connected components.
Additional difficulty in our work also arises from the fact that having a large connected component does not necessarily have to imply that the number of small boxes containing vertices of this set has to be of the same order, as we can have up to $\frac{\log n}{\log\log n}$ vertices in a set of constant volume. 

Furthermore, since we are not in the bounded degree setting, results like \Cref{prop:largeconnsetshavelargeboundaryanyd} and \Cref{prop:exponentialclusterrepulsionmodified} do not immediately follow from the corresponding statements with degree replaced by number of vertices.

In the case of two dimensions, in Section~\ref{sec:d=2isop} we also give a different and neater, but less general proof, for which we require the radius $r$ to be large enough and which is motivated by Benjamini and Mossel~\cite{benjamini_percolation}. A suitable renormalisation argument might lead to a proof for any $r$ and possibly higher dimensions, but we did not explore this direction further.

To prove the lower bound on the mixing time, we show that the relaxation time is of order $n^{\frac2d}$.

From the variational characterisation of the spectral gap, we have
\[
\trel\ge \frac{1}{1-\lambda_2}=\sup_{f\text{ is non-constant}}\frac{\Var_{\pi}(f)}{\mathcal{E}(f)},
\]
where $\mathcal{E}(f)=\frac{1}{2}\sum_{x,y}\pi(x)P(x,y)(f(x)-f(y))^2$ is the Dirichlet form of $f$. We will thus get a lower bound by just picking a suitable function $f$.

Note that the upper bound on the mixing time in terms of the isoperimetric profile requires the chain to be lazy. Therefore, we will first show all of our results for the lazy version of the chain. We will then use recent results by Hermon, Langer and Malmquist, which imply that the mixing time of a random walk on a graph is of the same order as the mixing time of the lazy version as long as the graph is not close to being bipartite in a certain sense.

Notation: For functions $a,b:\Znonneg\to[0,\infty)$ we write $a(n)\lesssim b(n)$ if there exists a constant $C>0$ such that for all sufficiently large $n$ we have $a(n)\le Cb(n)$, and write $a(n)\ll b(n)$ if for any constant $c>0$, for all sufficiently large $n$ (in terms of $c$) we have $a(n)\le cb(n)$. We define $\gtrsim$ and $\gg$ analogously. We write $a(n)\asymp b(n)$ if we have $a(n)\lesssim b(n)\lesssim a(n)$. Unless specified otherwise, these relations are considered as $n\to\infty$.

We end this section by stating two useful lemmas which control the degrees of the giant component of RGG. The first one upper bounds the sum of degrees of RGG in boxes of side length $(\log n)^{\frac2d}$, and follows using \cite[Proposition~7.1]{rgg_too_many_edges}. Similar statements have been proved in previous works, but for completeness, we present the proof in Appendix~\ref{appendix:degreebounds}. 

\begin{lemma}\label{lemma:sumsofdeginboxes} If we cover \(\Lambda_n^{(d)}\) by \(2^d n\) boxes of side length \((\log n)^{\frac{2}{d}}\) (\(n\ge 4\)), with high probability the sum of the degrees of the vertices in each box is less than \(2\gamma_d r^d (\log n)^2\), where  $\gamma_d$ is the volume of the unit ball in $d$ dimensions.
\end{lemma}

The giant component \(L_1(G)\) is with high probability the unique component of \(G\) with \(|L_1(G)|=\Theta(n)\). We will also use \cite[Lemma 2.10]{tcov_rgg}, which states the following.

\begin{lemma}\label{lemma:ordernedges} There exist constants \(\til c\) and \(\til C\) such that, with high probability, \(\til c n \le |E(L_1(G))| \le \til C n\). 
\end{lemma}

\section{Isoperimetric bound for \texorpdfstring{$d\ge 2$}{d ge 2}}\label{sec:isopbounds}

The proof presented in this section is motivated by the work of Pete~\cite{pete_exp_cluster_rep}. 

\begin{definition}[Boxes and blocks]~\label{def:boxes_blocks} Let \(L\in\mathbb N\). Suppose \(n\ge(r_{\max}L)^d\), and consider a regular grid \(s\Z^d/{\sim}\) on \(\Lambda_n^{(d)}\) with \(s\in\left[r(n),2r(n)\right]\) and \(\frac{\sqrt[d]n}{Ls}\in \mathbb N\).
Call the cells of this grid \emph{$s$-boxes} and define \emph{blocks}
\[B_{\frac{3Ls}{4}}(Lx)=\left\{y\in \Lambda_n^{(d)} \,\,\colon\, ||y-Lx||_\infty \le \frac{3Ls}{4}\right\},\]
for \(x\in s\Z^d/{\sim}\).
Say blocks \(B_{\frac{3Ls}{4}}(Lx)\) and \(B_{\frac{3Ls}{4}}(Ly)\) are \emph{adjacent} if \(||\frac xs-\frac ys||_1 \le 1\) and \emph{*-adjacent} if \(||\frac xs-\frac ys||_\infty \le 1\).
For example, the highlighted blocks in \Cref{fig:boxesAndBlocks} are not adjacent but *-adjacent.
\end{definition}

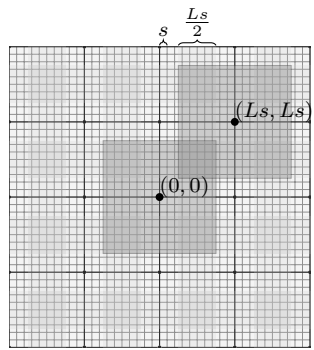
\begin{figure}[htb]
\centering
\begin{tikzpicture}[node font=\scriptsize]
    \begin{scope}
    \clip (0,0) rectangle (4,-4);

    \coordinate (B0) at (2,-2);
    \coordinate (B1) at (3,-1);
    
    % Background grids and faint blocks
    \draw[help lines, step=0.1, line width=0.01pt] (0,0) grid (4,-4);
    \graph[simple, empty nodes, nodes={circle, draw, fill=black, inner sep=0pt, minimum size=1pt}]{
        {[grid placement] subgraph Grid_n[n=25]};
    };
    \foreach \x in {-1,...,5}{
        \foreach \y in {1,...,-5}{
            \fill[gray!40!white,opacity=0.2] ($(\x,\y)+(-0.75,-0.75)$) rectangle ++(1.5,1.5);
        }
    }

    % Highlighted blocks
    \draw[gray,draw=black,fill=gray,opacity=0.3] ($(B0)+(-0.75,-0.75)$) rectangle ++(1.5,1.5);
    \draw[gray,draw=black,fill=gray,opacity=0.3] ($(B1)+(-0.75,-0.75)$) rectangle ++(1.5,1.5);

    \end{scope}
    % Labels
    \fill (2,-2) circle (1.5pt) node[above right=-0.15]{$(0,0)$};
    \fill (3,-1) circle (1.5pt) node[above right=-0.15]{$(Ls,Ls)$};
    \draw[decoration={brace,raise=0.5pt,amplitude=2}, decorate] (2,0) -- node[above=1pt]{$s$} (2.1,0);
    \draw[decoration={brace,raise=0.5pt,amplitude=2}, decorate] (2.25,0) -- node[above=1pt]{$\frac{Ls}{2}$} (2.75,0);

    \path (-0.2,0) -- (4.2,0); % centering
\end{tikzpicture}%
\caption{Illustration of $s$-boxes and blocks on \(\Lambda_n^{(d)}\). Here \(d=2\), \(L=10\) and \(\frac{\sqrt[d]n}{Ls}=4\).}
\label{fig:boxesAndBlocks}
\end{figure}

\begin{definition}[Good blocks]~\label{def:good_blocks} Call a block \(B\) \emph{good} if
\begin{itemize}
    \item it has a cluster with vertices within \(r\) of all its \((d-1)\)-dimensional faces, and
    \item all its other clusters have \(\ell_\infty\)-diameter \(<\frac{Ls}5\).
\end{itemize}
If the block is not good, we call it \emph{bad}.
\end{definition}

We start by giving some simple properties of sets of bad blocks. In particular, by comparison to Bernoulli site percolation, we upper bound the probability of the existence of a large set of bad blocks which has a sufficiently small number (in terms of its size) of *-connected components.

\begin{claim}\label{claim:closed_blocks} For any \(d\ge2\), for any $p\in (0,1)$, there is a large enough constant $C_p$ such that for all $L\ge C_p$ the set of good blocks stochastically dominates a Bernoulli site percolation of the blocks with parameter $p$. Therefore, for $L\ge C_p$ the probability that in a fixed set of $k$ blocks, all of the blocks are bad is at most $(1-p)^k$. 
\end{claim}

\begin{proof} From~\cite[Lemma 3.1]{tcov_rgg} (or~\cite[Theorem~2]{penrose_pisztora_large_dev_cts_perc}) it follows that  the probability that a block is good goes to $1$ as \(L\to\infty\). As non-*-adjacent blocks are good independently,
using~\cite[Theorem~7.65]{grimmett_percolation} we may, for any \(p\in(0,1)\), take \(L\) sufficiently large that the set of good blocks stochastically dominates Bernoulli site percolation with parameter \(p\).
\end{proof}

\begin{claim}\label{claim:numberofsizeksetswithnottoomanystarcomps}
    For any \(d\ge2\), there exists a constant \(C\) such that the following holds for all \(n\) and \(k\).
    A regular \(d\)-dimensional toroidal grid with \(\le n\) vertices
    has \(\le C^k\) subsets of size \(k\) with \(\le \frac{k}{\log n}\) *-connected components.
\end{claim}
The proof of this claim is a simple counting, so we defer it to Appendix~\ref{appendix:countingclaim}.

\begin{lemma}\label{lemma:bigbadwithfew*conncompsunlikely}
    For any \(d\ge 2\), for sufficiently large \(L\), \(n\ge(r_{\max}L)^d\) and \(k\ge0\),
    \begin{equation}
    \label{eq:bigbadwithfew*conncompsunlikely}
    \pr{\exists\, \text{ set of bad blocks on }\Lambda_n^{(d)}\text{ of size } k \text{ with }\le \frac{k}{\log n} \text{ *-connected components}} \le \frac{1}{2^k}.
\end{equation}
\end{lemma}

\begin{proof}
Fix \(d\ge2\).
Let \(p=1-\frac{1}{2C}\), with \(C\) from \Cref{claim:numberofsizeksetswithnottoomanystarcomps}.
Take \(L\ge \max(\frac{1}{r_{min}},\, C_p)\), with \(C_p\) from \Cref{claim:closed_blocks}.
Then by \Cref{claim:closed_blocks} the probability that a fixed set of \(k\) blocks is a set of bad blocks is at most \(\left(\frac{1}{2C}\right)^k\),
and by \Cref{claim:numberofsizeksetswithnottoomanystarcomps} there are \(\le C^k\) sets of \(k\) blocks on \(\Lambda_n^{(d)}\) with \(\le \frac{k}{\log n}\) *-connected components.
So a union bound gives, for any \(n\ge(r_{\max}L)^d\) and \(k\),
\begin{equation*}
    \label{eq:bigBadUnlikely}
   \pr{\exists\, \text{ set of bad blocks on }\Lambda_n^{(d)}\text{ of size } k \text{ with }\le \frac{k}{\log n} \text{ *-connected components}}\le C^k\left(\frac{1}{2C}\right)^k = \frac{1}{2^k},
\end{equation*}
as desired. \end{proof}

\begin{definition}[Substantial blocks]\label{def:substantial_blocks}Given any subset \(S\) of \(G\), call a block \emph{\(S\)-substantial} if
\begin{itemize}
    \item it has a cluster of \(S\) of \(\ell_\infty\)-diameter \(\ge\frac{Ls}5\).
\end{itemize}
Write \(S^L\) for the set of \(S\)-substantial blocks.
\end{definition}

We note that if \(S\) is connected, then \(S^L\) is connected.

\begin{definition}[Vertex boundaries]\label{def:vertex_boundaries} For a subset $S$ of $G$, define the inner vertex boundary of $S^L$, $\operatorname{\partial_{iV}}{\!\left(S^L\right)}$, to be the set of blocks that are $S$-substantial but have an adjacent block that is not $S$-substantial, i.e.\ \(\operatorname{\partial_{iV}}{\!\left(S^L\right)}=\{B\in S^L \colon \exists \til B \in (S^L)^c \,\text{ s.t. }B \text{ is adjacent to }\til B\}\). Similarly, define the outer vertex boundary to be the set of blocks which are in $(S^L)^c$ but have an adjacent block in $S^L$. 

For subsets $S$ and $T$ of $G$, define the inner vertex boundary of \(S^L\) towards \(T^L\), \(\operatorname{\partial_{iV}^{T^L}}{\!\left(S^L\right)}\), as follows: 
Let \(W_1,\,\dots,\,W_m\) be the *-connected components of \(\left(S^L\right)^c\) that contain or have in their outer vertex boundary at least one block of \(T^L\).
Then \(\operatorname{\partial_{iV}^{T^L}}{\left(S^L\right)} := \bigcup_{j=1}^m \partial_{oV}(W_i)\).
\end{definition}

It is easy to see that $\operatorname{\partial_{iV}}{\!\left(S^L\right)}$ consists of bad blocks whenever \(L\ge\frac{10}{3}\), and so does \(S^L \cap T^L\) when $S$ and $T$ belong to different clusters.
Clearly, \(\operatorname{\partial_{iV}^{T^L}}{\left(S^L\right)} \subset \operatorname{\partial_{iV}}{\!\left(S^L\right)}\).

\begin{definition}[Boundary boxes between two sets]\label{def:boundarysboxes} For any sets $S$ and $T$, write \(\tau(S,\,T)\) for the set of $s$-boxes that contain a point of \(S\) and whose corresponding $5s$-boxes contain a point of \(T\).
Write \(\upsilon(S,\,T)\) for the set of $s$-boxes in \(\tau(S,\,T)\) whose corresponding $3s$-boxes contain a point of $\operatorname{\partial_{oV}}{\!\left(S\right)}$.
Write \(|\tau(S,\,T)|\) and \(|\upsilon(S,\,T)|\) for the number of $s$-boxes in these sets.
\end{definition}

We denote by $\deg(A)=\sum_{x\in A}\deg(x)$ the sum of degrees of vertices of $A$ and by $\comps(A)$ the set of connected components of $A$.

\begin{proposition}\label{prop:exponentialclusterrepulsionmodified}
    For any \(d\ge 2\), for any sufficiently large \(L\), there exist constants \(c\) and \(N\) such that for all \(n\ge N\), \(m\ge (\log n)^{2d}\) and \(t\), the following holds. For $A,B\subset G$ let the event $Y_{A,B}(m,t)$ be the event that the following hold: 
    \begin{enumerate}
    \item $A$ and $B$ belong to different clusters of $G$
    \item both $A$ and $B$ have total degrees at least $m$, i.e.\ $\deg(A)\ge m$ and $\deg(B)\ge m$
    \item $A$ is connected
    \item all components of $B$ have total degrees at least $(\log n)^{2d}$, i.e.\ $\min_{D\in \comps(B)}\deg(D)\ge (\log n)^{2d}$
    \item $A$ and all components of $B$ have \(\ell_\infty\)-diameter at least $\frac{2}{5}Lr_{\max}$, i.e.\ $\min_{D\in \{A\}\cup\>\!\comps(B)}\diam{D}\ge
            \frac{2}{5}Lr_{\max} $ 
    \item $|\tau(A,B)|\ge t$.
    \end{enumerate}
    
    Then we have that
    \begin{equation*}
        \pr{\exists A,B \subset G: Y_{A,B}(m,t)}\le e^{-c\max(m^{1-\frac1d},\, t)}.
    \end{equation*}
\end{proposition}

\begin{proof} Fix \(d\ge2\) and \(L\ge30\) satisfying the conditions in the proof of Lemma~\ref{lemma:bigbadwithfew*conncompsunlikely}.

If \(A\), \(B\) are subsets of \(G\) satisfying the conditions, the set
\begin{align*}
\mathcal P :=
\left(A^L \cap B^L\right)
\cup \operatorname{\partial_{iV}^{B^L}}{\!\left(A^L\right)}
\end{align*}
is a set of bad blocks.
Indeed, this is because $\operatorname{\partial_{iV}^{B^L}}{\!\left(A^L\right)}\subset \operatorname{\partial_{iV}}{\!\left(A^L\right)}$ which only contains bad blocks and as $A^L\cap B^L$ contains blocks which are substantial for two sets which belong to different clusters, which means that they are also bad.
So by Lemma~\ref{lemma:bigbadwithfew*conncompsunlikely} it is sufficient to prove that \(|\mathcal P| \ge c_1 \max(m^{1-\frac1d},\, t)\), for some constant \(c_1>0\), and that each *-connected component of \(\mathcal P\) has size \(\ge \log n\).

%For the \(t\)-part of the bound:
Each $s$-box in \(\tau(A,B)\) is contained in at least one block of \(A^L\cap B^L\),
since \(A\) and all components of \(B\) have \(\ell_\infty\)-diameter \(\ge\frac{Ls}5\) and \(\frac{Ls}5+\frac{Ls}5+3s \le \frac {Ls}2\).
Each block contains \(\le (3L)^d\) $s$-boxes.
Hence \(|\mathcal P| \ge \frac{1}{(3L)^d}t\).

% For the \(m\)-part of the bound:
Subdivide the torus into boxes of side length between \(\frac{(\log n)^{\frac{2}{d}}}{2}\) and \((\log n)^{\frac{2}{d}}\).
Then \(A^L\) and \(B^L\) each intersect at least \( c_2\frac{m}{(\log n)^2}\) such \emph{log-boxes}, for some constant \(c_2\), since \(A\) and all components of \(B\) have \(\ell_\infty\)-diameter \(\ge\frac{Ls}5\) and using \Cref{lemma:sumsofdeginboxes}, $A$ and $B$ each intersect $\gtrsim \frac{m}{(\log n)^2}$ such boxes. Assume that $|A^L\cap B^L|\le \frac{c_2}{4}m^{1-\frac{1}{d}}$, as otherwise we already have a wanted lower bound on $|\mathcal{P}|$.
Take \(N\) large enough that \(\frac{(\log n)^{\frac{2}{d}}}{2}\ge 3Lr_{\max}\).
Write \(W_1,\,\dots,\,W_m\) for the *-connected components of \(\left(A^L\right)^c\) that contain or have in their outer vertex boundary at least one block of \(B^L\).

If \(\ge \frac{c_2}4 m^{1-\frac1d}\) log-boxes intersect \(\operatorname{\partial_{iV}^{B^L}}{\!\left(A^L\right)}\), then
\[|\mathcal P|\ge \left|\operatorname{\partial_{iV}^{B^L}}{\!\left(A^L\right)}\right| \ge \frac{c_2}{4\cdot2^d}m^{1-\frac1d}.\]
If not, then \(\ge c_2\frac{m}{(log n)^2}-\frac{c_2}{2}m^{1-1/d}\ge\frac{c_2}{2}\frac{m}{(log n)^2}\) log-boxes intersect \(A^L\) but not \(\operatorname{\partial_{iV}^{B^L}}{\!\left(A^L\right)}\) or $B^L$, and \(B^L\) but not $A^L$ (and therefore not \(\operatorname{\partial_{iV}^{B^L}}{\!\left(A^L\right)}\)), respectively.
We may add the blocks that are fully contained in such log-boxes to \(A^L\) and \(B^L\), respectively, without changing \(\mathcal P\).
Indeed, the affected blocks belonged to neither \(A^L\) nor \(B^L\) and are added to only one of them. The blocks added to $A^L$ did not belong to \(W_1\cup...\cup W_m\) which means that the connected component of $(A^L)^c$ to which they belonged did not have any blocks from $B^L$ which means that adding them could not have increased the outer boundary towards $B^L$. The blocks added to $B^L$ belonged to \(W_1\cup...\cup W_m\) so adding them to $B^L$ could also not have increased the size of $\operatorname{\partial_{iV}^{B^L}}{\!\left(A^L\right)}$.
Label the corresponding sets of blocks after this addition \(\widehat{A}^L\) and \(\widehat{B}^L\), and note that as they now fully contain $\gtrsim \frac{m}{(\log n)^2}$ boxes of side length $\gtrsim (\log n)^{\frac{2}{d}}$ we have
\begin{equation}
    |\widehat{A}^L|,\,|\widehat{B}^L| \ge c_3m,\label{eq:largeALBL}
\end{equation}
for some constant \(c_3>0\).
We may assume $|A^L\cap B^L|\le \frac{c_3}{2}m$.
The grid-isoperimetric inequality for the torus \cite{bollobas_leader_isoperimetric} gives
\begin{align*}
|\mathcal P|
&\ge \,c_4\min\left(\frac{1}{2^d}\sum_i |W_i|^{1-\frac1d},\, |\widehat{A}^L|^{1-\frac1d}\right)
&\ge \,c_4\min\left(\frac{1}{2^d}\left(\sum_i|W_i|\right)^{1-\frac1d},\, |\widehat{A}^L|^{1-\frac1d}\right) 
&\ge \frac{c_3^{1-\frac1d}\frac{c_4}{2^d}}{2}m^{1-\frac1d},
\end{align*}
where the last inequality follows by \eqref{eq:largeALBL} and \(|\widehat{B}^L|=|A^L\cap B^L|+\sum_i |W_i\cap \widehat{B}^L|\).

It remains to prove that each *-connected component of \(\mathcal P\) has size \(\ge \log n\).

If \(\Lambda_n^{(d)}\) was a box, \cite[Theorem 3]{boundary_connectivity} would give that each \(\partial_{oV}(W_i)\) is *-connected.
\(\Lambda_n^{(d)}\) is a torus, and bicoherence of the torus \cite[Lemma 9.2]{penrose_rgg} gives that each \(\partial_{oV}(W_i)\) has \(\le2\) *-connected components.
But then on the torus, when there are two such components, both must contain \(\gtrsim n^{\frac1d} \gg \log n\) blocks.

The contribution \(\mathcal P_j:=\left(A^L \cap B_j^L\right) \cup \operatorname{\partial_{iV}^{B_j^L}}{\!\left(A^L\right)}\) of a component \(B_j\) of \(B\) to \(\mathcal P\) is either *-connected or each of its *-connected components contains a component of a not *-connected \(\partial_{oV}(W_i)\).
Indeed, if \(A^L\cap B_j^L=\emptyset\), then \(\mathcal P_j=\partial_{oV}(W_i)\) for some \(i\) and the previous statement clearly holds.
Otherwise, we can construct a sequence of sets containing \(A^L\cap B_j^L\) starting from \(B_j^L\) by, for each \(i\) for which $B^L_j\cap \partial_{oV}(W_i)\ne \emptyset$ in turn, replacing the blocks in \(W_i\) with \(\partial_{oV}(W_i)\).
\(B^L_j\) is *-connected, and since \(B^L_j\cap \partial_{oV}(W_i)\subset A^L\cap B_j^L\), the number of *-connected components can increase only where \(\partial_{oV}(W_i)\) is disconnected.
If \(A^L\cap B_j^L\neq\emptyset\), the last set in the sequence is $\mathcal P_j$, as then $B^L_j\cap \partial_{oV}(W_i)\ne \emptyset$ whenever $B^L_j\cap (W_i\cup \partial_{oV}(W_i))\ne \emptyset$ because $B^L_j$ is connected and intersects $A^L$.
So the previous statement holds in this case too.

If $\mathcal P_j$ is not *-connected, each of its *-connected components has size $\gg \log n$ (as they contain one of the two parts of some not *-connected $\partial_{oV}(W_i)$). If $\mathcal P_j$ is *-connected, by applying the previous part to $A^L$ and $B_j^L$ (with $m=(\log n)^{2d}$) we have  \(|\mathcal P_j|\ge c_1(\log n)^{2(d-1)}\), and so it is a *-connected component of size at least $\log n$ and the proof follows.
\end{proof}

\begin{proof}[Proof of Proposition~\ref{prop:largeconnsetshavelargeboundaryanyd}]
For any \(A\) in \(L_1(G)\), let \(B(A)\) be the union of the components \(B_i\) of \(A^c\setminus B_{\frac{3s}{2}}(\upsilon(A,A^c))\) that have \(\deg(B_i) \ge (\log n)^{2d}\), where by \(A^c\) we mean the complement of $A$ in the giant, i.e.\ \(L_1(G)\setminus A\).
Let \(M:=C\til c(\log n)^{2d^2}\), for some large constant \(C\). Let \(\hat c\) be a small constant.
Define $X$ to be the event that $ \deg(G) \le 2\gamma_d r_{\max}^d n$, that the total degree in any $\frac25Lr_{\max}$-box is less than $(\log n)^{2d}$, i.e.\  $\max_{b\in\{\frac25Lr_{\max}\text{-boxes}\}}\deg(b)<(\log n)^{2d}$, and that there exists $A\!\subset\! L_1(G)$ for which the following holds:

\begin{enumerate}
\item $  \deg(A) \ge M$, and $ \deg(A^c) \ge \deg(A)$
\item $A$ and $A^c$ are both connected
\item $|\upsilon(A,\, B(A))| < \hat c(\deg A)^{1-\frac1d}$.
\end{enumerate}

For each \(m\ge M\) and \(t\ge0\), define $Y(m,t)$ to be the event that $Y_{\widetilde{A},B}(\frac{m}{2},t)$ (from Proposition~\ref{prop:exponentialclusterrepulsionmodified}) holds for some $\widetilde{A}$ which is a cluster in $G$ with $\deg(\widetilde{A})=m$ and let $ Y:=\bigcup_{m=M}^\infty \bigcup_{t=0}^\infty Y(m,t)$.

Define a map  $F:\,X\to Y$ mapping $ \omega \mapsto \omega'$ as follows:
For each configuration \(\omega\in X\), choose a set \(A\) satisfying the defining conditions of \(X\). Let \(\omega'\) be the configuration obtained from \(\omega\) by deleting the points in \(B_{\frac{3s}{2}}(\upsilon(A,\,B(A)))\setminus A\).
To see \(\omega'\in Y\), let \(\til A\) be the connected component of \(A\) \footnote{which might be larger than $A$ as it can contain small components of $A^c\setminus B_{\frac{3s}{2}}(\upsilon(A,\,B(A)))$} in \(\omega'\) and let \(B\) be \(B(\til A)\) in \(\omega\).
Then \(B(A)\subset B\), and connectedness of \(A^c\) in \(\omega\) implies
\(\numcomps(A^c\setminus B_{\frac{3s}{2}}(\upsilon(A,\,B(A))))\le c_1 |\upsilon(A,B(A))|\), for some constant \(c_1=c_1(d,r_{\max})\).
So
\begin{align*}
    \deg(B)
    &\ge\deg(A)-(\log n)^{2d}c_1\hat c (\deg (A))^{1-\frac1d}\\
    &\ge\deg(\til A)-2\cdot(\log n)^{2d}c_1\hat c (\deg (A))^{1-\frac1d}\\
    &\ge\frac{\deg(\til A)}{2},
\end{align*}
where for the last inequality we have assumed $C$ is sufficiently large.  Also, by the definition of $B(\widetilde{A})$, all components in it have a total degree which is lower bounded by $(\log n)^{2d}$. It follows that, under this assumption, \(\til A\) and \(B\) in \(\omega'\) satisfy the defining conditions of \(Y\).

\Cref{prop:exponentialclusterrepulsionmodified} gives
\begin{align*}
    \pr{Y(m,t)}\le e^{-\frac{\hat{c}_1}{2}\max(m^{1-\frac1d},\, t)},
\end{align*}
where \(\hat{c}_1\) is a constant from the proposition.

For any \(\omega'\) in the image of \(F\), its preimages in \(X\) can be recovered as follows: Choose as \(\til A\) one of the \(\le 2\gamma_d r_{\max}^d n\) components in \(\omega'\) with \(\deg(\til A)\ge M\). Let \(B\) be the union of the other components in \(\omega'\) that have \(\deg(\cdot)\ge (\log n)^{2d}\).
Assume \(\til A\) was chosen in such a way that \(\til A\) and \(B\) satisfy the defining conditions of \(Y\).
Write \(m:=\deg(\til A)\) and \(t:=|\tau(\til A, B)|\).
Add (zero or more) points to \(t\land \lfloor \hat c m^{1-\frac1d}\rfloor\) of the \(t\) $3s$-boxes \(B_{\frac{3s}{2}}(\tau(\til A,B))\). By \cite[Theorem 9.4]{penrose_last_ppp_lectures}, for \(\til u\) a subset of \(\tau(\til A,B)\), a PPP on \(B_{\frac{3s}{2}}(\til u)\) conditioned to have points at \(\til A\cap B_{\frac{3s}{2}}(\til u)\) is distributed as the original PPP plus the appropriate points.
So, since \(\pr{K\ge 0}\le e^{\lambda} \pr{K=0}\) when \(K\sim \Pois{\lambda}\),
\begin{equation}
\label{eq:prXsumbound}
\pr{X}=\pr{F^{-1}(Y)}
\le \sum_{m=M}^{\infty} \sum_{t=0}^{\infty} 2\gamma_d r_{\max}^d n \binom{t}{t\land \lfloor \hat c m^{1-\frac1d}\rfloor} Q^{t\land \lfloor \hat c m^{1-\frac1d} \rfloor} e^{-\frac{\hat{c}_1}{2}\max(m^{1-\frac1d},\, t)}
\end{equation}
where \(Q>1\) is a constant depending on $r_{\max}$, which is an upper bound on the inverse of the probabilities that a specific box in $\tau(\tilde{A},B)$ has no added points.

Bounding the right hand side of \eqref{eq:prXsumbound}, following the argument of Pete~\cite{pete_exp_cluster_rep}, we get, for any $K\ge 1$ 
\[
P(X)\le 2\gamma_dr_{\max}^dn\sum_{m=M}^\infty\left(\sum_{t=0}^{K\hat{c}m^{1-\frac{1}{d}}} \binom{t}{t\land \lfloor \hat c m^{1-\frac1d}\rfloor}Q^{t}e^{-\frac{\hat{c}_1}{2}m^{1-\frac{1}{d}}}+\sum_{t=K\hat{c}m^{1-\frac{1}{d}}+1}^\infty \binom{t}{\lfloor \hat c m^{1-\frac1d}\rfloor}Q^{\lfloor \hat c m^{1-\frac1d}\rfloor}e^{-\frac{\hat{c}_1}{2}t}\right)
\]

and by making $K$ sufficiently large, for all $\hat{c}$ sufficiently small in terms of $\hat{c}_1$ and $K$, the second sum can be upper bounded by $\sum_{t=K\hat{c}{m^{1-\frac{1}{d}}}+1}^\infty e^{-\frac{\hat{c}_1}{4}t}\le \exp\left(-K\hat{c}\frac{\hat{c}_1}{4}m^{1-\frac{1}{d}}\right)$, while the terms in the first sum are upper bounded by $(1+Q)^{2t}e^{-\frac{\hat{c}_1}{2}m^{1-\frac{1}{d}}}$ and so for all $K$ and sufficiently small $\hat{c}$  it is also bounded by $\exp\left(-\frac{\hat{c}_1}{4}m^{1-\frac{1}{d}}\right)$. 
Hence, for \(\hat c\) chosen sufficiently small,
\begin{align}
\label{eq:stp1}
\pr{X} \le 2\gamma_d r_{\max}^d n\,e^{-c_2M^{1-\frac1d}}\underset{n\to\infty}{\longrightarrow} 0.
\end{align}

If there is a set \(A\subset L_1(G)\) with both $A$ and $A^c=L_1(G)\setminus A$ connected, $\frac12\ge\pi(A)\ge \frac{C(\log n)^{2d^2}}{n}$ and less than \(c (\deg A)^{1-\frac1d}\) edges between \(A\) and \(A^c\) for a small enough constant $c$ (in terms of $\hat{c}$ and $d$), then the defining conditions of \(X\) are satisfied unless some box \(b\) of side length \(\frac{2}{5}Lr_{\max}\) has \(\deg(b) \ge (\log n)^{2d}\), \(\deg(G)>2\gamma_d r_{\max}^d n\) or \(\deg(L_1(G))<\til c n\). By \eqref{eq:stp1}, \Cref{lemma:sumsofdeginboxes}, a simpler bound and \Cref{lemma:ordernedges}, the probability of each of these events goes to \(0\) as \(n\to\infty\), so with high probability no such \(A\) exists.
\end{proof}

\section{Bounds on the mixing and relaxation time}

We now give bounds on mixing time and on the relaxation time. 

\begin{proof}[Proof of Theorem~\ref{trm:tmixRGG} and Proposition~\ref{prop:trelRGG}]
With high probability, by \Cref{prop:largeconnsetshavelargeboundaryanyd} for all \(A\) which are connected and have connected complement and with $\frac12\ge\pi(A)\ge \frac{C(\log n)^{2d^2}}{n}$, we have
\begin{equation*}
\label{eq:phiaLowerBoundBig}
\Phi(A) = \frac{Q(A,A^c)}{\pi(A)}
\ge \frac{c(\deg(A))^{1-\frac{1}{d}}}{\deg(L_1(G))\pi(A)}= \frac{c(\pi(A))^{1-\frac{1}{d}}}{(\deg(L_1(G)))^{\frac{1}{d}}\pi(A)}
\overset{\ref{lemma:ordernedges}}\ge \frac{c}{(2\til C)^{\frac1d}} \frac{1}{({n\,\pi(A)})^{\frac{1}{d}}}.
\end{equation*}
We work on the high probability event that the above holds.
For all \(A\) with $\pi(A)\le \frac{C(\log n)^{2d^2}}{n}$, using that \(L_1(G)\) is connected, we have
\begin{equation*}
\label{eq:phiaLowerBoundSmall}
\Phi(A) = \frac{Q(A,A^c)}{\pi(A)}
\ge \frac{n}{\deg(L_1(G))C(\log n)^{2d^2}}
\overset{\ref{lemma:ordernedges}}\ge \frac{1}{2\til C C} \frac{1}{(\log n)^{2d^2}}.
\end{equation*}

 Using \cite[Lemma 4.36]{aldous_fill_2014}, to obtain a lower bound on $\Phi_*=\min_{A:\pi(A)\le \frac{1}{2}}\Phi(A)$ we only need to consider connected $A$ with $A^c=L_1(G)\setminus A$ connected, as such sets achieve the minimum of the quantity $\frac{Q(A, A^c)}{\pi(A)\pi(A^c)}$. Using the above bounds, the bottleneck ratio therefore satisfies \[\Phi_*\ge \alpha\min\left(\frac{1}{n^{\frac{1}{d}}},\,\frac{1}{(\log n)^{2d^2}}\right) \] where we take $\alpha=\min\left\{\frac{c}{(2\til C)^{\frac1d}},\frac{1}{2\til C C} \right\}$.  Similarly to \cite[Lemma 2.6]{benjamini_percolation} we now show that, with high probability, for all $\frac12\ge x \ge \pi_{\min}$ we have
\[\Phi_*(x)\gtrsim \min\left(\frac{1}{(nx)^{\frac{1}{d}}},\,\frac{1}{(\log n)^{2d^2}}\right).\]
First, we fix a small constant $q<\frac14$ to be specified later and note that if $q<x$, the above holds as $\Phi_*(x)\ge \Phi_*\ge \alpha q^{\frac1d}\min\left(\frac{1}{(nq)^{\frac{1}{d}}},\,\frac{1}{(\log n)^{2d^2}}\right)$. Now consider $\pi_{\min}\le x\le q$ and let  $\Phi_*(x)=\Phi(A)$ where $A$ with $\pi(A)\le x$ is taken to be a connected set, as the infimum in the isoperimetric profile is achieved for some connected set. If $\pi(A)\le \frac{C(\log n)^{2d^2}}{n}$ the desired inequality clearly holds, so we assume the opposite. Likewise, assume that $\Phi(A)<\frac{\alpha}{(nx)^{\frac{1}{d}}}$, as otherwise the desired inequality holds. This means that $A^c=L_1(G)\setminus A$ can be split into connected components $A_1,\ldots, A_r$ for some $r>1$. Let $i\le r$ be such that $\pi(A_i)$ is maximal. If $\pi(A_i)\le \pi(A)$ then there is a set $B$ which is a union of $A$ and some of the sets $A_j$ for $j\le r$ such that $\pi(B)\in (\frac{1}{4},\frac{1}{2})$. We then have for all large enough $n$ that
\[ \frac{\alpha}{n^{\frac1d}}\le \Phi_*\le \Phi(B)\le \frac{Q(B,B^c)}{\pi(B)}\le 4Q(A,A^c)\le \alpha\frac{4x}{(nx)^{\frac{1}{d}}}\le\alpha \frac{4q^{1-\frac{1}{d}}}{n^{\frac1d}}\]
and so by taking $q$ sufficiently small, this gives a contradiction. 
Therefore $\pi(A_i)\ge \pi(A)$. Let $A'$ be the set of smaller $\pi$ measure of $A_i$ and $A_i^c$. Then $\pi(A)\le \pi(A')\le \frac{1}{2}$ and $A'$ is connected with connected complement. Therefore, $\Phi(A')\ge \frac{\alpha}{(n\pi(A'))^{\frac1d}}$ and so 
\[Q(A,A^c)\ge Q(A',A'^c)\ge \alpha \frac{\pi(A')^{1-\frac1d}}{n^{\frac1d}}\ge \alpha \frac{\pi(A)^{1-\frac1d}}{n^{\frac1d}}.\] This finally implies that \[\Phi_*(x)=\Phi(A)\ge \frac{\alpha}{(n\pi(A))^{\frac1d}}\ge   \frac{\alpha}{(nx)^{\frac1d}} \] completing the proof of the lower bound on the isoperimetric profile. 

With high probability, using the isoperimetric profile bound from~\cite{evolvingsets} we now have
\begin{equation*}
    \tmixtext
    \lesssim \int_{4\pi_{\min}}^{\frac12}\frac{\mathrm d x}{x\Phi_*^2(x)} + \Phi_*^{-2}
     \lesssim \int_{0}^{\frac12}n^{\frac{2}{d}}x^{1-\frac{2}{d}}\,\mathrm dx + \int_{\frac{4}{2\til C n}}^{\frac12} \frac{(\log n)^{4d^2}}{x}\,\mathrm dx + n^{\frac{2}{d}}\lesssim n^{\frac{2}{d}}
\end{equation*}
as desired. This completes the proof of the upper bound on the mixing time, and as the mixing time always upper bounds the order of the relaxation time, we also have an upper bound on the relaxation time. We now show a same order lower bound on the relaxation time, which then also lower bounds the mixing time by this same order and shows that there is no cutoff. 

For the lower bound, we use the variational characterisation of the spectral gap, which tells us that for this lazy reversible chain $\trel \ge \frac{\Var_{\pi}(f)}{\mathcal{E}(f)}$ for any non-constant function $f$. Consider a box of side length~$\frac{n^{\frac{1}{d}}}{10}$  centred at the origin and a diametrically opposite box to it on the torus of the same side length. We know that with high probability, both of these have a giant component, which has a total sum of the edges at least $\widetilde{c} {n}$ for a suitable constant $\tilde{c}$. With high probability, the second-largest component of RGG has size~$\ll n$, implying that both of these components belong to $L_1(G)$. Then there is a constant $\widetilde{c}_1$ such that at least $\widetilde{c}_1$ proportion of edges are joining the vertices with $\ell_\infty$ distance from the origin at most $\frac{1}{5}n^{\frac{1}{d}}$ while at least the same proportion of edges are joining the vertices with $\ell_\infty$ distance from the origin at least $\frac{2}{5}n^{\frac{1}{d}}$.  Similarly to \cite{benjamini_percolation}, we can use the function $f(x)=\|x\|_\infty$ for which we now have that $\Var_\pi(f)\ge  \frac{\widetilde{c}_1}{100}n^{\frac{2}{d}}$ while the $\mathcal{E}(f)\le r_{\max}^2$ as vertices with $P(x,y)>0$ have distance $\le r_{\max}^2$. This shows that $\trel\ge\frac{\widetilde{c}_1n^{\frac{2}{d}}}{100r_{\max}^2}$, completing the proof of the lower bound for the relaxation time, and therefore also for the mixing time of the lazy simple random walk, and establishes that there is no cutoff. 

We now show the same results for the simple random walk on RGG. A result from \cite{LRP} which relies on the results from \cite{near_bipartiteness} states that for a simple random walk on a graph $\trel\asymp \frac{1}{1-\lambda_2}$ holds if there exists a constant $\delta$ such that for any partition of vertices into set $A$ and $B$ we have \begin{equation}\label{eq:nearbipartite}\left|\left\{\{x,y\}\in E: (x,y)\in (A\times A){\bigcup} (B\times B)\right\}\right|\ge \delta \sum_{x\in V}\deg(x).\end{equation}  
Mix-hit comparison from \cite[Proposition 1.8 and Remark 1.9]{characterisation_of_cutoff} implies that, as the relaxation times for lazy simple random walk and simple random walk on this graph are of the same order, and hit times for the two can easily be bounded in terms of each other, then the order of the mixing time for the simple random walk is the same as for the lazy simple random walk and so the proof follows upon showing \eqref{eq:nearbipartite}. It is easy to see that on the high probability event that the sum of degrees in RGG is upper bounded by $Cn$ to get that every partition of $L_1(G)$ satisfies~\eqref{eq:nearbipartite}, it is enough to know that there are at least $\gtrsim n$ non-intersecting triangles in $L_1(G)$. To this end, it is enough to know that if we subdivide the torus into boxes of side length $s\le \frac{r}{10}$, the number of vertices in $L_1(G)$ which belong to the boxes containing at least $3$ vertices of $L_1(G)$ is $\gtrsim 1$. Indeed, if such a box has $k$ vertices, it has at least $\lfloor k/3\rfloor$ disjoint triangles, as any $3$ vertices make a triangle and also triangles belonging to different boxes are clearly disjoint (as we can assume that no points of the PPP are on the boundary of the boxes). Notice that as $r_{\min}>r_g$ we can take a small enough $\varepsilon$ such that RGG corresponding to a PPP with intensity $1-\varepsilon$ (and same radius $r$) has a giant component which, with high probability, has at least $m \asymp n$ vertices. If the total number of different $s$-boxes which contain at most $2$ vertices of this giant is $\le \frac{m}{4}$, then at least $\frac{m}{2}$ vertices already belong to $s$-boxes which have at least three vertices of this giant. Otherwise, for each of at least $\frac{m}{4}$ boxes with at most two vertices of the giant, there is a positive probability $p$ that the $s$-box has at least $2$ vertices of the PPP with intensity $\varepsilon$ and due to the concentration of independent Bernoulli trials, with high probability at least $\frac{mp}{8}$ of such boxes will have at least three vertices belonging to the union of these two PPPs (which is a PPP of intensity $1$). In both cases, we get $\asymp n$ vertices belong to $s$-boxes with at least three vertices in them, and the proof follows.
\end{proof}

\section{Isoperimetric bound for \texorpdfstring{$d=2$}{d=2}}\label{sec:d=2isop}

In this section, we present an easier proof of the isoperimetric bound for the case of $d=2$ and large enough $r_{\min}$, which follows the work of Benjamini and Mossel~\cite{benjamini_percolation}. This bound also holds for sets with $\frac12 \ge \pi(A)\ge \frac{C(\log n)^2}{n}$, which makes an improvement compared to \Cref{prop:largeconnsetshavelargeboundaryanyd}.

\begin{proposition}\label{prop:largesetshavelongdistance} For \(d=2\), there exist constants $C$ and $c$ such that, with high probability, each $A$ with $\pi(A)\ge \frac{C(\log n)^2}{n}$ contains two vertices at distance at least $c\sqrt{\deg(A)}$.
\end{proposition}
\begin{proof}
Let \(\til c\), \(\til C\) be as in~\Cref{lemma:ordernedges}.
Write \(D=\max(2\gamma_d r_{\max}^d, 2\til C)\).
Choose \(C=\frac{16D}{2\til c}\) and \(c=\frac{1}{4\sqrt {D}}\).
Subdivide the torus \(\Lambda_n^{(2)}\) into boxes of side length \(\ell\in[\frac{\log n}{2},\,\log n]\).
By \Cref{lemma:sumsofdeginboxes} and the pigeonhole principle,  with high probability  each \(A\) intersects at least \(\frac{\deg(A)}{D(\log n)^2}\) boxes.
So with high probability each non-empty \(A\) contains two points at horizontal or vertical distance at least
\[\min\left(\frac{\sqrt n-r_{\max}}{2},\, \left(\sqrt{\frac{\deg(A)}{D(\log n)^2}}-2\right)\ell\right).\]
For $A$ with $\pi(A)\ge \frac{C(\log n)^2}{n}$, by \Cref{lemma:ordernedges}, with high probability \(\frac{\deg(A)}{D(\log n)^2}\ge \frac{C2\til c}{D}=16\). So, taking \(n\ge 4r_{\max}^2\), with high probability for each such \(A\) we have that the above is at least 
\[
 \frac14 \min\left(\sqrt n,\, \sqrt{\frac{\deg(A)}{D}}\right) = c\sqrt{\deg(A)},\]
where the last equality follows from \Cref{lemma:ordernedges} using \(D\ge2\til C\).
\end{proof}

\begin{proposition}\label{prop:largeconnsetshavelargeboundary} For \(d=2\), there exist positive constants $C$ and $c$ such that if \(r_{\min} \ge 4\), % 4 is larger than r_g (d=2)
the following holds with high probability.
For all connected $A$ with $\frac12\ge\pi(A)\ge \frac{C(\log n)^2}{n}$ the number of edges from \(A\) to \(A^c\) is at least $c\sqrt{\deg(A)}$.
\end{proposition}

\begin{proof}
Consider a regular grid on \(\Lambda_n^{(2)}\) of side length \(s\in\left[\frac{2}{\sqrt{13}}r,\,\sqrt{\frac25}\,r\right]\).
Each grid box is split by diagonals into four triangles, and we say that each triangle corresponds to the grid edge which it contains as a side.
Call a grid edge \emph{open} if there are points in both of its triangles.
An example is shown in \Cref{fig:percExample}.
This gives a \(\Ber{p}\) bond percolation, \(p =  (1-e^{-\frac{s^2}{4}})^2 \ge (1-e^{-\frac{r^2}{13}})^2\).
Assume \(r_{\min} > \sqrt{13\log(2+\sqrt{2})} \approx 3.995\), so that \(p > p_c = \frac12\).
For a path $\Gamma$ of grid edges we denote by $|\Gamma|$ the number of edges in that path.
Then by a theorem of Kesten \cite[Theorem 1]{kesten_large_dev_passage_time}, there exist constants \(a\) and \(C_1\) such that
\begin{align}
    \label{eq:long_path_positive_proportion_open}
    \pr{\forall  \Gamma \text{ with } |\Gamma|\ge \alpha\log n\text{, } \Gamma \text{ has } \ge a\cdot |\Gamma|\text{ open edges}}\ge
    1-\left(\frac{\sqrt{n}}{s}\right)^2\cdot 2e^{-C_1 \alpha\log n}.
\end{align}
For any constant \(\alpha>\frac{1}{C_1}\), the right hand side goes to \(1\) as \(n\to\infty\).
In particular, this tells us that there exists a constant $\alpha$ such that with high probability any path of length at least $\alpha\log n$ has at least some constant proportion of edges which are open.

\begin{figure}[htb]
\centering
\begin{tikzpicture}[font=\scriptsize]
    \begin{scope}
    \clip (0,0) rectangle (3,-3);
    \graph[simple, empty nodes, nodes={circle, draw, fill=black, inner sep=0pt, minimum size=1pt}]{
        {[grid placement] subgraph Grid_n[n=16]};
        % 1 -- [bend right] 2; % 4 -!- 8 -!- 12 -!- 16; % 1 -!- 2 -!- 3 -!- 4;
    };
    \coordinate (m1) at ($(2)!.5!(7)$);
    \coordinate (m2) at ($(6)!.5!(11)$);
    
    \draw[gray] (11) -- ($(2)!.5!(6)$);
    \draw[gray] (11) -- (m1);
    \draw[gray] (10) -- (11);

    \draw plot[mark=*,mark size=1pt, only marks] coordinates {
        (1.3,-0.8)
        (1.4,-1.1)
        (1.840576,-0.037255)
        (2.025000,-0.903449)
        (2.321730,-1.054476)
        (2.555099,-2.107670)
        (1.473781,-1.670222)
        (2.731213,-2.288483)
        (0.446671,-2.323859)
        (0.796457,-0.386927)
    };
    \draw[line width=1.5pt, red] (6) -- (7);
    \draw[line width=0.6, dashed] (6) -- (m2) -- (7) -- (m1) -- (6) -- (7);
    %\draw[dashed, gray, line width = 0.3] (11) -- (m1);
    %\draw[dashed, gray, line width = 0.3] (11) -- ($(2)!.5!(6)$);

    \path[gray] (11) -- node[below left=-8pt,yshift=0.5pt]{$\frac{\sqrt{13}}{2}s$} ($(2)!.5!(6)$);
    \path[gray] (11) -- node[below right=-8pt,xshift=4pt,yshift=1pt]{$\sqrt{\frac52}s$} (m1);
    \path[gray] (10) -- (11);
    \draw[gray, decoration={brace,raise=0.5pt,amplitude=3,mirror}, decorate] (10) -- node[below=1pt]{$s$} (11);

    \end{scope}
\end{tikzpicture}%
\caption{Illustration of boxes, triangles and openness. Here the red edge is open, the other grid edges are not. The condition \(s\in\left[\frac{2}{\sqrt{13}}r, \sqrt{\frac25}\,r\right]\) ensures \(r\) is bounded by the lengths of the oblique gray segments.}
\label{fig:percExample}
\end{figure}
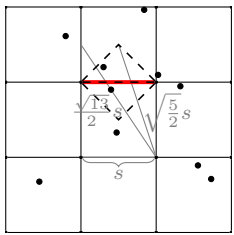

Choose a constant \(C > \max\left(\hat C,\, \frac{64r_{\max}^2}{5\til c \hat c^2 C_1^2}\right)\), where \(\til c\) is as in~\Cref{lemma:ordernedges} and \(\hat c\), \(\hat C\) are the constants \(c\), \(C\) from \Cref{prop:largesetshavelongdistance}.

Let \(A\) be connected with $\frac12 \ge  \pi(A) \ge \frac{C(\log n)^2}{n}$.
Call a box \emph{red} if it contains at least one point of \(A\), and otherwise call it \emph{white}.
Consider the white *-connected components, by which we mean maximal connected subgraphs in a graph where boxes are connected if they share at least one grid vertex (i.e.\ if they share a corner or an edge).

If at least $\varepsilon \sqrt{n}$ of such components contain a point of $A^c$, then the number of edges between $A$ and $A^c$ is at least $\frac{\varepsilon \sqrt{n}}{25}$.
Indeed, as \(L_1(G)\) is connected, each such component which has a point of $A^c$ also has a point of $A^c$ which is a neighbour in $L_1(G)$ either of some point in a red box or of a point in another component of white boxes.
If it is connected to a point in the red box, if that point is in $A$, this gives a connection between $A$ and $A^c$ corresponding to this component and if the point is in $A^c$, as we know that red boxes have at least one point of $A$ and all points in $s$-boxes are connected in $G$, there is a connection between $A$ and $A^c$ within this red box corresponding to this component of white boxes.
If the component only has a connection in $L_1(G)$ to another component of white boxes, then this connecting edge has to intersect a perpendicular bisector of a red box, but then, due to simple geometric bounds illustrated in \Cref{fig:edgeRadiusAndInnerCross},
one of the endpoints of the connecting edge has to be connected to a vertex in the red box in $A$ (which exists as the box is red), and that gives a connection between $A$ and $A^c$ corresponding to this component.
We see that we have identified a corresponding connection for each component of white boxes and that each of these connections can correspond to at most $25$ different components.

\begin{figure}[htb]
\centering
\begin{subfigure}{0.4\textwidth}
    \centering
    \begin{tikzpicture}[font=\scriptsize]
    \pgfmathsetmacro{\r}{sqrt(5/2)}
    \pgfmathsetmacro{\edger}{(sqrt 3)/2*\r}
    \pgfmathsetmacro{\rMinusEps}{sqrt(5/2)-0.4}

    \fill[gray!10!white] (\rMinusEps, \edger) arc (90:-90:\edger) -- (0,-\edger) arc (270:90:\edger) -- cycle;

    \draw[color=gray] (0.5*\rMinusEps,0) -- node[right=-3pt]{{$\frac{\sqrt 3}{2}r$}} (.5*\rMinusEps,\edger);
    \draw[color=gray] (-\r,0) -- node[below]{{$r$}} (0,0);
    \draw[color=gray] (\rMinusEps,0) -- node[below]{{$r$}} (\r+\rMinusEps,0);

    \draw plot[mark=*, mark size=1pt] coordinates {(0,0) (\rMinusEps,0)} circle (\r);
    \draw (0,0) circle (\r);
    \path (0,0) node[below]{$x$} -- node[below]{$\le r$} (\rMinusEps,0) node[below]{$y$};
\end{tikzpicture}%
    \caption{If $xy\in E$, $z\in V$ and \(d(xy, z) \le \frac{\sqrt 3}{2}r\), then $xz\in E$ or $yz\in E$.}
    \label{fig:edgeRadius}
\end{subfigure}
\quad
\begin{subfigure}{0.4\textwidth}
    \centering
    \begin{tikzpicture}
    \pgfmathsetmacro{\r}{sqrt(5/2)}
    \pgfmathsetmacro{\edger}{(sqrt 3)/2*\r}

    \begin{scope}
    \clip (0,0) rectangle (3,-3);

    \graph[simple, empty nodes, nodes={circle, draw=none, fill=none, inner sep=0pt, minimum size=1pt}, edges={draw=none}]{
        {[grid placement] subgraph Grid_n[n=16]};
    }; % not drawn

    \coordinate (m1) at ($(2)!.5!(7)$);
    \coordinate (m2) at ($(6)!.5!(11)$);
    \coordinate (m3) at ($(5)!.5!(10)$);
    \fill[red!20!white] (9) rectangle (6);
    \fill[red!20!white] (10) rectangle (7);
    \fill[red!20!white] (6) rectangle (3);
    \fill[black!80!red,opacity=0.15] (6) rectangle (3);

    \draw plot[mark=*,mark size=1pt] coordinates {(0.85,-0.8) (2.15,-1.05)};
    \path (0.85,-0.8) node[left]{$x$} -- (2.13,-1.17) node[right]{$y$};

    \draw[gray, line width=0.6, dashed] ($(m1)+(-.5,0)$) -- ++(1,0);
    \draw[gray, line width=0.6, dashed] ($(m1)+(0,-.5)$) -- ++(0,1);
    \draw[gray, line width=0.6, dashed] ($(m2)+(-.5,0)$) -- ++(1,0);
    \draw[gray, line width=0.6, dashed] ($(m2)+(0,-.5)$) -- ++(0,1);
    \draw[gray, line width=0.6, dashed] ($(m3)+(-.5,0)$) -- ++(1,0);
    \draw[gray, line width=0.6, dashed] ($(m3)+(0,-.5)$) -- ++(0,1);

    \graph[simple, empty nodes, nodes={circle, draw, fill=black, inner sep=0pt, minimum size=1pt}]{
        {[grid placement] subgraph Grid_n[n=16]};
    };
    \end{scope}
    \begin{scope}
        \clip (0,0) rectangle (3.7,-3); % careful not to clip too much
        \draw[color=gray] ($(m1)+(0,-.5)$) -- node[right=-1pt]{{$\frac{\sqrt 5}{2}s< \frac{\sqrt 3}2r$}} (3);
    \end{scope}
\end{tikzpicture}%
    \caption{If $xy\in E$ intersects a perpendicular bisector (dashed line) of a box \(B\) and $z\in V\cap B$, then, from (a), $xz\in E$ or $yz\in E$.}
    \label{fig:edgeRadiusInnerCross}
\end{subfigure}
\caption{Elementary geometric observations used in the proof of~\Cref{prop:largeconnsetshavelargeboundary}.}
\label{fig:edgeRadiusAndInnerCross}
\end{figure}
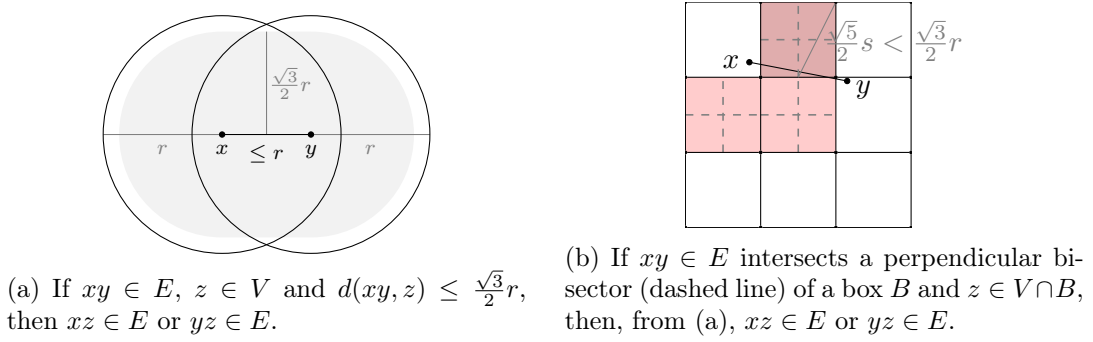

We may assume that at most $\varepsilon \sqrt{n}$ white *-connected components contain a point of $A^c$, so that the total \(\pi\)-mass of white components with \(\pi\)-mass smaller than \(\frac{C(\log n)^2}{n}\) is smaller than \(\eps\frac{C (\log n)^2}{\sqrt n} < \eps\), for \(n\) sufficiently large.
Indeed, \Cref{lemma:ordernedges} gives with high probability \(\deg(A) \le \deg(L_1(G)) \le 2\til C n\), so otherwise with high probability we have a lower bound $\frac{\varepsilon \sqrt{n}}{25}\ge \frac{\eps}{25\sqrt{2\til C}}\sqrt{\deg(A)}$ on the number of edges between $A$ and $A^c$.

Similarly, if \(\pi(A^c\text{ in red boxes})\geq \eps\) then with high probability the total degree of vertices in $A^c$ which are in red boxes would be at least $\til c\eps n$ and as with high probability the maximal degree in $L_1(G)$ is smaller than $\til c \sqrt{n}$ we would have at least $\eps\sqrt n$ vertices in $A^c$ in red boxes, but all of them would have a neighbour in $A$ inside their red box, and we would have at least $\varepsilon \sqrt{n}$  edges between $A$ and $A^c$. Therefore, we may assume \(\pi(A^c\text{ in red boxes})<\eps\).
So, taking $\eps\le \frac16$, a $\pi$-mass of at least \(\pi(A^c)-2\eps\ge \frac12 - 2\eps \ge \frac13\) comes from white components of \(\pi\)-masses \(\ge \frac{C(\log n)^2}{n}\).
Denote these components \(W_1, \dots, W_m\).

Let $W^+$ be the set of white boxes for which some edge between two points of \(A\) intersects their perpendicular bisector.
Let $\mathcal E^+$ be the set of grid edges of these boxes.
Fix \(i\in\{1,\dots,m\}\).
Let \(\mathcal E_i = \partial W_i \cup (\mathcal E^+\cap W_i)\), i.e.\ the set of grid edges which either are between a white box of $W_i$ and a red box or in \(\mathcal E^+\) and on $W_i$.

As a closed subset of \(\Lambda_n^{(2)}\), \(\mathcal E_i\) is the intersection of closed connected subsets which have union \(\Lambda_n^{(2)}\).
Indeed, we may take $W_i$ and $Z_i$, where $Z_i$ is the union of $\mathcal E^+$ and the boxes not in \(W_i\).
Note $W_i$ is connected as a closed subset of $\Lambda_n^{(2)}$ because it is *-connected as a set of boxes,
and $Z_i$ is connected because $A$ is connected in \(L_1(G)\), cf.\ \Cref{fig:addingGoodEdges}.
So by bicoherence of the torus \cite[Lemma 9.2]{penrose_rgg}, \(\mathcal E_i\) has at most two connected components.

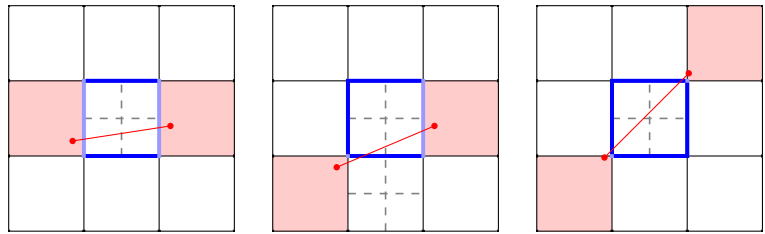
\begin{figure}[htb]
\centering
\begin{tikzpicture}
    \begin{scope}
    \clip (0,0) rectangle (3,-3);
    \graph[simple, empty nodes, nodes={circle, draw=none, fill=none, inner sep=0pt, minimum size=1pt}, edges={draw=none}]{
        {[grid placement] subgraph Grid_n[n=16]};
    }; % not drawn
    \coordinate (m1) at ($(2)!.5!(7)$);
    \coordinate (m2) at ($(6)!.5!(11)$);
    \fill[red!20!white] (9) rectangle (6);
    \fill[red!20!white] (11) rectangle (8);
    \draw[gray, line width=0.6, dashed] ($(m2)+(-.5,0)$) -- ++(1,0);
    \draw[gray, line width=0.6, dashed] ($(m2)+(0,-.5)$) -- ++(0,1);
    \graph[simple, empty nodes, nodes={circle, draw, fill=black, inner sep=0pt, minimum size=1pt}]{
        {[grid placement] subgraph Grid_n[n=16]};
    };
    \draw[line width=1.5pt, blue!40!white] (10) rectangle (7);
    \draw[line width=1.5pt, blue] (10) -- (11);
    \draw[line width=1.5pt, blue] (7) -- (6);
    
    \draw[color=red] plot[mark=*,mark size=1pt] coordinates {
        (0.85,-1.8)
        (2.15,-1.6)
    };
    
    % \draw[line width=0.6, dashed] (6) -- (m2) -- (7) -- (m1) -- (6) -- (7);
    \end{scope}
\end{tikzpicture}
\quad
\begin{tikzpicture}
    \begin{scope}
    \clip (0,0) rectangle (3,-3);
    
    \coordinate (m1) at ($(2)!.5!(7)$);
    \coordinate (m2) at ($(6)!.5!(11)$);
    \coordinate (m3) at ($(10)!.5!(15)$);
    \fill[red!20!white] (13) rectangle (10);
    \fill[red!20!white] (11) rectangle (8);
    \draw[gray, line width=0.6, dashed] ($(m2)+(-.5,0)$) -- ++(1,0);
    \draw[gray, line width=0.6, dashed] ($(m2)+(0,-.5)$) -- ++(0,1);
    \draw[gray, line width=0.6, dashed] ($(m3)+(-.5,0)$) -- ++(1,0);
    \draw[gray, line width=0.6, dashed] ($(m3)+(0,-.5)$) -- ++(0,1);
    \graph[simple, empty nodes, nodes={circle, draw, fill=black, inner sep=0pt, minimum size=1pt}]{
        {[grid placement] subgraph Grid_n[n=16]};
    };
    % \draw[line width=1.5pt, blue] (10) -- (11);
    \draw[line width=1.5pt, blue!40!white] (10) rectangle (7);
    \draw[line width=1.5pt, blue] (10) -- (11);
    \draw[line width=1.5pt, blue] (7) -- (1,.-1) -- (10);

    \draw[color=red] plot[mark=*,mark size=1pt] coordinates {
        (0.85,-2.15)
        (2.15,-1.6)
    };
    \end{scope}
\end{tikzpicture}
\quad
\begin{tikzpicture}
    \begin{scope}
    \clip (0,0) rectangle (3,-3);
    
    \coordinate (m1) at ($(2)!.5!(7)$);
    \coordinate (m2) at ($(6)!.5!(11)$);
    \fill[red!20!white] (13) rectangle (10);
    \fill[red!20!white] (7) rectangle (4);
    \draw[gray, line width=0.6, dashed] ($(m2)+(-.5,0)$) -- ++(1,0);
    \draw[gray, line width=0.6, dashed] ($(m2)+(0,-.5)$) -- ++(0,1);
    \graph[simple, empty nodes, nodes={circle, draw, fill=black, inner sep=0pt, minimum size=1pt}]{
        {[grid placement] subgraph Grid_n[n=16]};
    };
    % \draw[line width=1.5pt, blue] (10) -- (11) -- (7);
    \draw[line width=1.5pt, blue!40!white] (10) rectangle (7);
    \draw[line width=1.5pt, blue] (10) -- (2,-2) -- (7);
    \draw[line width=1.5pt, blue] (7) -- (1,-1) -- (10);
    
    \draw[color=red] plot[mark=*,mark size=1pt] coordinates {
        (0.9,-2.02)
        (2.02,-0.9)
    };
    \end{scope}
\end{tikzpicture}%
\caption{The three cases where we use the additional edges.
Suppose two red boxes in different *-connected components are connected by an \(L_1(G)\)-edge $e$ with endpoints in \(A\).
Then \(e\) must intersect a perpendicular bisector of a white cell \(c\) (in the third case, this follows by \(r\le  \frac{\sqrt{13}}{2}s\)). So \(c\in W^+\) and the edges of \(c\), shown in blue, are in \(\mathcal E^+\).}
\label{fig:addingGoodEdges}
\end{figure}

Write \(L_i=\min(\sqrt{\deg(A)},\,\sqrt{\deg(W_i)})\).
In this paragraph we show with high probability \(\mathcal E_i\) contains three edges of pairwise distances at least \(\frac{\hat c}8L_i\).
Suppose \(n\) is large enough that \(\frac{\hat c}{8}L_i\ge 4s + r\).
By \Cref{prop:largesetshavelongdistance}, we can pick $x$ and $y$ in \(A^c \cap W_i\) with \(||x-y||_2\ge \hat cL_i\).
Without loss of generality, suppose that the distance in the first coordinate between some $x$ and $y$ is at least $\frac{\hat c}2L_i$.
Since \(W_i\) is *-connected, we can pick a third point \(z\) in \(W_i\) with distance in the first coordinate to $x$ and to $y$ at least $\frac{\hat c}4L_i$.
If the box column of \(x\) or one of its neighbouring columns contains a red box and the analogous statements hold for \(y\) and \(z\), the claim holds.
So we may assume \(W_i\) contains all of two adjacent columns.
But then we can also pick \(\til x,\til y,\til z\) in \(W_i\) with distances in the second coordinate at least  $\frac{\hat c}4L$, so we may also assume \(W_i\) contains all of two adjacent rows.
This makes $W_i^c$ a subset of some box, and applying \Cref{prop:largesetshavelongdistance} again, this time to \(A\), and using that \(A\) is connected in \(L_1(G)\), gives the claim.

Since two of these three edges must belong to the same component of \(\mathcal E_i\), \(\mathcal E_i\) contains a path \(\Gamma_i\) with at least \(\frac{\hat c}{8s}L_i\) edges.
By \eqref{eq:long_path_positive_proportion_open}, as \(C\) was chosen sufficiently large, with high probability at least a constant proportion \(a\) of these edges are open.
But when an edge \(e\in \mathcal E_i\) is open, i.e.\ when there are points in both of its triangles, there is an edge from a point of \(A^c\cap W_i\) assigned to $e$ to a point of \(A\), cf.\ Figures~\ref{fig:percExample} and~\ref{fig:edgeRadiusInnerCross}.

So with high probability there are at least \(\frac{a\hat c}{8s}L_i\) edges from \(A\) to \(A^c\cap W_i\), giving at least
\begin{align*}
  \frac{a\hat c}{8s}\min\left(\sqrt{\deg(A)},\,\sum_i \sqrt{\deg(W_i)}\right)
  & \ge \frac{a\hat c}{8s}\min\left(\sqrt{\deg(A)},\,\sqrt{\sum_i \deg(W_i)}\right)
  \ge c\sqrt{\deg(A)}
\end{align*}
edges from \(A\) to \(A^c\), where \(c:=\frac{a\hat c}{8s_{\max}\sqrt{3}}\), \(s_{\max}=\sqrt{\frac25}r_{max}\).

\end{proof}

\appendix 

\section{Proof of Lemma~\ref{lemma:sumsofdeginboxes}}\label{appendix:degreebounds}

\begin{proof}[Proof of Lemma~\ref{lemma:sumsofdeginboxes}]
Let \(B_1,\dots,B_{2^dn}\) be a cover of \(\Lambda_n^{(d)}\) by boxes of side length \((\log n)^{\frac{2}{d}}\).
For each \(i\), define the inflated box \(\til B_i = B_i+(2^{\frac{1}{d+1}}-1)\cdot(B_i-\mathrm{centre}(B_i))\) and let \(\til G_i\) be RGG (on a torus) with fundamental domain \(\til B_i\).
For \(n\) sufficiently large we have \(\frac{(2^{\frac{1}{d+1}}-1)(\log n)^{\frac{2}{d}}}{2} > r\), so that
\begin{equation}
\label{eq:degBiBoundedByDoubleTorusEdgeCount}
\sum_{x\in B_i} \deg_G(x) \le \sum_{x\in \til B_i} \deg_{\til G_i}(x)=2e(\til G_i).
\end{equation}

Write \(|E|:=e(\til G_i)\) and \(|V|:=|\til G_i|\). By \cite[Proposition~7.1]{rgg_too_many_edges}, for any \(\delta>0\),
\[
\lim_{n\to\infty}
\frac{\log \pr{|E| > (1+\delta)\E{|E|}}}{\sqrt{\E{|E|}}\log \E{|V|}} = -\sqrt{\frac \delta2}.
\]

So for any fixed \(\eps > 0\), for all \(n\) sufficiently large,
\begin{equation}
\label{eq:edgeuppertailbound}
\pr{|E| > (1+\delta)\E{|E|}} < e^{-(1-\eps)\sqrt{\frac \delta2} {\sqrt{\E{|E|}}\log \E{|V|}}}.
\end{equation}

For \(n\) sufficiently large, we have
\[\sqrt{\E{|E|}}\log \E{|V|} = \sqrt{\frac{\gamma_d r^d}{2}2^{\frac{d}{d+1}}(\log n)^2} \log\left(2^{\frac{d}{d+1}}(\log n)^2\right)\asymp \log n \log(\log n)^2.\]

Plugging this into~\eqref{eq:edgeuppertailbound} with \(\delta=2^{\frac{1}{d+1}}-1\) and, say, \(\eps=\frac12\), we find
\begin{equation}
\pr{2e(\til G_i) > 2\E{\sum_{x\in B_i} \deg_G(x)}} =  \pr{e(\til G_i) > 2^{\frac{1}{d+1}}\frac12\gamma_d r^22^{\frac{d}{d+1}}(\log n)^2} < n^{-2}
\end{equation}
for all \(n\) sufficiently large.
Recalling \eqref{eq:degBiBoundedByDoubleTorusEdgeCount}, the claim now follows by a union bound over the \(2^d n\) boxes.
\end{proof}

\section{Proof of Claim~\ref{claim:numberofsizeksetswithnottoomanystarcomps}}\label{appendix:countingclaim}
\begin{proof}[Proof of Claim~\ref{claim:numberofsizeksetswithnottoomanystarcomps}]
Assume without loss of generality that \(1\le \log n \le k \le n\).
By a Peierls argument \cite[Lemma~9.3]{penrose_rgg}, we may choose a constant \(C_1=C_1(d)\) such that for all \(k\), \(\Z^d\) has \(\le C_1^k\) *-connected subsets of size \(k\) containing \(0\). Given a regular \(d\)-dimensional toroidal grid with \(\le n\) vertices, the following procedure can produce any subset of size \(k\) with \(\le \frac{k}{\log n}\) *-connected components:
\begin{enumerate}
    \item\label{item:s1} Choose an integer \(r\) between \(1\) and \(\frac{k}{\log n}\).
    \item\label{item:s2} Choose distinct points \(\{x_i\}_{i=1}^r\) of the grid.
    \item\label{item:s3} Choose positive integers \(\{k_i\}_{i=1}^r\) with \(\sum_{i=1}^r k_i=k\).
    \item\label{item:s4} Choose *-connected subsets \(\{A_i\}_{i=1}^r\) of the grid, with \(x_i\in A_i\) and \(|A_i|=k_i\) for all \(i\).
    \item The produced subset is \(A:=\bigcup_{i=1}^r A_i\).
\end{enumerate}

The combined number of choices for steps~\ref{item:s1} and~\ref{item:s2} is
\begin{equation*}
    \sum_{r=1}^{\left\lfloor\frac{k}{\log n}\right\rfloor} \binom{n}{r} \le
    \sum_{r=1}^{\left\lfloor\frac{k}{\log n}\right\rfloor} n^r \le
    \frac{k}{\log n} n^{\frac{k}{\log n}} \le
    e^{2k}.
\end{equation*}

The number of choices for step~\ref{item:s3} is (by stars and bars) $\binom{k-1}{r-1}\le 2^{k-1}$. The number of choices for step~\ref{item:s4} is \(\le \Pi_{i=1}^r C_1^{k_i} = C_1^k\). Hence, the claim holds with \(C=e^22C_1\).
\end{proof}

\subsection*{Acknowledgements} We thank Perla Sousi for suggesting this question and for helpful discussions. This work was supported by the Cambridge Summer Research in Mathematics Bursary Fund and St Edmund's College Arie Hector Award. 

\bibliography{RGG.bib}
\bibliographystyle{plain}

\end{document}